\documentclass[12pt]{amsart}

\usepackage{amsfonts}
\usepackage{mathrsfs}
\usepackage{graphicx}
\usepackage{amssymb, comment}
\usepackage[margin=1in]{geometry}
\usepackage[all]{xy}
\usepackage{tikz,everypage}

\newtheorem{thm}{Theorem}[section]
\newtheorem{lemma}[thm]{Lemma}
\newtheorem{cor}[thm]{Corollary}

\newtheorem{claim}[thm]{Claim}
\newtheorem{conj}[thm]{Conjecture}

\theoremstyle{definition}
\newtheorem{defn}[thm]{Definition}
\newtheorem{example}[thm]{Example}

\theoremstyle{remark}

\newcommand{\abs}[1]{\left\lvert#1\right\rvert}

\newcommand{\MCG}{\mathcal{M}\mathrm{od}}	
\newcommand{\HG}{\mathcal{H}}				

\newcommand{\QHS}{\mathcal{V}}				
\newcommand{\BQHS}{\mathcal{B}}		
\newcommand{\QH}{\mathrm{Q}\mathrm{H}}		
\newcommand{\TG}{\mathcal{T}}    		
\newcommand{\JG}{\mathcal{J}}    		

\newcommand{\bbc}{\mathbb{C}}

\newcommand{\bbr}{\mathbb{R}}
\newcommand{\bbz}{\mathbb{Z}}
\newcommand{\calc}{\mathcal{C}}
\newcommand{\cald}{\mathcal{D}}
\newcommand{\calt}{\mathcal{T}}

\title[Quasi-homomorphisms on mapping class groups]{Quasi-homomorphisms on mapping class groups vanishing on a handlebody group}

\author{Jiming Ma}
\address{School of Mathematical Sciences, Fudan University, Shanghai, 200433, P. R. China}
\email{majiming@fudan.edu.cn}

\author{Jiajun Wang}
\address{LMAM, School of Mathematical Sciences, Peking University, Beijing, 100871, P. R. China}
\email{wjiajun@math.pku.edu.cn}

\begin{document}


\maketitle

\begin{abstract} We construct infinitely many linearly independent quasi-homomorphisms on the mapping class group of a Riemann surface with genus at least two which vanish on a handlebody subgroup. As a corollary, we disprove a conjecture in \cite{Reznikov:2000hq}. Another corollary is that there are infinitely many linearly independent quasi-invariants on the Heegaard splittings of three-manifolds.
\end{abstract}

\section{Introduction}

For a group $G$, a function $\varphi:G\rightarrow\bbr$ is called a \emph{quasi-homomorphism} if
$$\abs{\varphi(ab)-\varphi(a)-\varphi(b)}\leq K,\quad \forall a,b\in G,$$
for some fixed $K\in\bbr_{\ge 0}$, and the minimum of $K$  in above inequality is called the \emph{defect} of $\varphi$. Let $\QHS(G)$ be the vector space of all quasi-homomorphisms on $G$. A bounded function is a quasi-homomorphism, and the bounded quasi-homomorphisms form a subspace of $\QHS(G)$, denoted by $\BQHS(G)$. The quotient $\QH(G):=\QHS(G)/\BQHS(G)$ is called the \emph{quasi-homomorphism space} on $G$. For a subgroup $G_1<G$, we denote by $\QH(G, G_1)$ the subspace of $\QH(G)$ in which every element has a representative vanishing on $G_1$.

A quasi-homomorphism $\varphi:G\rightarrow\bbr$ is \emph{homogeneous} if $\varphi(a^m)=m\varphi(a)$ for all $a\in G$. For any quasi-homomorphism $\varphi:G\rightarrow\bbr$, its \emph{homogenization} $\widetilde{\varphi}$ defined by
$$\widetilde\varphi(a)=\lim_{m\rightarrow\infty}\frac{\varphi(a^m)}m,\quad\forall a\in G$$
is homogeneous. We have $[\varphi]=[\widetilde\varphi]$ in $\QH(G)$ and moreover, $\widetilde\varphi$ is \emph{pseudo-homomorphic} in the sense that $\widetilde\varphi(bab^{-1})=\widetilde\varphi(a)$ for any $a,b\in G$. If $\varphi$ is bounded on a subgroup $G_1$, then $\widetilde\varphi$ is vanishing on $G_1$.

Quasi-homomorphisms are related to bounded cohomology, bounded generation and stable commutator length (see  Bestvina-Fujiwara \cite{Bestvina:2002dr} and Calegari \cite{Calegari:2009ia}). In this note, we are interested in the relationship between quasi-homomorphisms and 3-manifolds. So, we are interested on mapping class groups of surfaces.

Let $S=S_{g,b}$ denote an oriented surface of genus $g$ with $b$ punctures. The mapping class group $\MCG(S)$ is the set of isotopy classes of orientation-preserving homeomorphisms, that is $\MCG(S)=\pi_0(\mathrm{Diff}^+(S))$ \cite{Farb:2012}. We require $3g+b>3$, it is easy to see the set of homomorphisms  from $\MCG(S)$ to $\bbr$ has zero or one dimension. However Bestvina-Fujiwara  \cite{Bestvina:2002dr} showed that  $\QH(\MCG(S))$ has infinite dimension, answer a conjecture of Morita \cite{Morita:1999cr}  (see also \cite{Kotschick:2001ia} for related result).

 Since  the braid group $B_n$ is  the mapping class group of a disk with $n$ marked points,   Bestvina and Fujiwara \cite{Bestvina:2002dr}  showed that  $\QH(B_n)$ has infinite dimension as well.  Moreover, some functions in  $\QH(B_n)$  have topological meanings: taking the signature of the link obtained by closing a braid gives a quasi-homomorphism on $B_n$ (Gambando and Ghys \cite{Gambaudo:1999dw}); the Rasmussen $s$ invariant \cite{Rasmussen:2010kg} from Khovanov homology theory and the Ozsv\'ath-Szab\'o $\tau$ invariant \cite{Ozsvath:2003vb} from Heegaard Floer theory both give quasi-homomorphisms on $B_n$ (Brandenbursky \cite{Brandenbursky:2011ia}), and their homogenizations are both the link number of a braid (divided by $2$ for $\tau$);  the Dehornoy floor of a braid $\beta \in B_n$ also gives a quasi-homomorphism on $B_n$ (Malyutin \cite{Malyutin:2005jq}).  From our viewpoint, the Dehornoy floor for braids is more interesting as quasi-homomorphisms on $B_n$  since it is not a link  invariant, in fact, it degenerates to zero when stabilizing $\beta \in B_n$  to $\beta \delta_{n}^{\pm} \in B_{n+1}$, where $\delta_{n}$ is the additional generator from $B_{n}$ to $B_{n+1}$.  Moreover, the Dehornoy floor  is not an invariant for conjugate classes, but it is coarsely defined for conjugate classes with defect 1 (Malyutin \cite{Malyutin:2005jq}). The Dehornoy floor can be used to distinguish $n$-braids with same closure (as links in $\mathbb{S}^{3}$).

\subsection{Main result}
We use standard terms for Heegaard splittings, see \cite{Scharlemann:2002} for details.

We hope quasi-homomorphisms can shed some light on 3-manifolds or Heegaard splittings of 3-manifolds as in the case of quasi-homomorphisms on  braid groups.

For a handlebody $H$ of genus $g>1$ with boundary $S$, the \emph{handlebody group} $\HG$ is the subgroup of $\MCG(S)$ that extends over $H$. We define
$$\QH(\MCG(S),\HG)=\{\alpha\in\QH(\MCG(S))\ \big|\ \text{there exists }f\in\alpha\text{ with } f\equiv0\text{ on } \HG\}.$$
We show the following

\begin{thm}\label{thm:MCH_H_infinite} For a handlebody $H$ of genus $g>1$ with boundary $S$, let $\HG$ be the handlebody group, then
$$\dim\QH(\MCG(S),\HG)=\infty.$$
\end{thm}

Theorem \ref{thm:MCH_H_infinite} evidently follows from

\begin{thm} \label{thm:main} Let $S$ be the boundary of a handlebody $H$ of genus $g>1$, for every  finite collection of cyclic subgroups $C_{1}, C_{2}, \ldots, C_{n} < \MCG(S)$, there exists an infinite dimensional subspace of  $\QH(\MCG(S))$ such that each $h$ in this subspace has the following properties:

\begin{enumerate}
\item  $h$ is bounded on each $C_{i}$;
\item  $h$ is bounded on the stabilizer of every essential curve in $S$;
\item $h$ is bounded on the handlebody group $\HG$.
\end{enumerate}

\end{thm}

\vspace{10pt}

\subsection{An application: Reznikov's conjecture}

A group $G$ is \emph{boundedly generated} if there are finitely many elements $g_1,\cdots,g_N$ in $G$, such that every element $g$ in $G$ can be written as $g_1^{m_1}\cdots g_N^{m_N}$,  for some integers $m_1,\cdots,m_N$. Tavgen showed in \cite{Tavgen} that most (and conjecturally all) non-cocompact lattices in higher rank simple Lie groups are boundedly generated. Farb, Lubotzky and Minsky showed in \cite{Farb:2001ka} that word hyperbolic groups and $\MCG(S_g)$ with $g\geqslant1$ are not boundedly generated. It is believed that $\MCG(S)$ shares some properties with lattices in a semisimple Lie
group. In \cite[Section 13]{Reznikov:2000hq}, the following conjecture is made:

\begin{conj}[Bounded width in Heegaard splittings]\label{conj:bounded_width} For any $g>1$, there are finitely many Dehn twists $\delta_1,\delta_2,\cdots,\delta_N$ on $S_g$ such that any double coset in
$$\HG\diagdown \MCG(S_g)\diagup\HG$$
contains an element of the form $\delta_1^{m_1}\delta_2^{m_2}\cdots\delta_N^{m_N}$  for some integers $m_1,\cdots,m_N$.
\end{conj}

Note that each double coset in the Conjecture \ref{conj:bounded_width} corresponds to a Heegaard splitting, and the conjecture asserts that the genus $g$ Heegaard splittings is (a quotient of) the product of finitely many cyclic groups.

\begin{thm} Conjecture \ref{conj:bounded_width}  is false.
\end{thm}

\begin{proof} Suppose that there exist Dehn twists $\delta_1,\delta_2,\cdots,\delta_N$ along curves $c_1,c_2,\cdots,c_N$ such that any $f \in \MCG(S)$ can be written as $f=a \delta_1^{m_1}\delta_2^{m_2}\cdots\delta_N^{m_N}b$ for some $a,b \in \HG$. By Theorem \ref{thm:main}, there exists a nonzero vector $h\in\QH(\MCG(S))$ that is bounded on the stabilizers of $c_1,\cdots,c_N$ and the handlebody group $\HG$. Hence $h(f)$ is uniformly bounded and $h$ is zero in $\QH(\MCG(S))$, a contradiction.
\end{proof}

In Reznikov's conjecture each $\delta_i$ is a Dehn twist along a simple closed curve. It is a natural question to ask whether Conjecture \ref{conj:bounded_width} holds if we take the $\delta_i$'s to be pseudo-Anosov maps. However, from (1) and (3) of Theorem \ref{thm:main}, this is also not true. So we disprove Reznikov's conjecture in a general sense.

\subsection{Generalizations to the Johnson filtrations and the Casson invariant}

Let $G=\pi_1(S_g)$ be the fundamental group of $S=S_g$. The mapping class group $\MCG(S)$ naturally acts on the lower central series of $G$,
$$G_1 >  G_2 >  G_3 > \cdots>  G_n > \cdots,$$
where $G_1=G$ and $G_n=[G_{n-1},G]$. Consequently, $\MCG(S)$ acts on the nilpotent quotient $G/G_k$, hence we get a homomorphism
$$\rho_k:\MCG(S)\rightarrow \mathrm{Aut}\left(G/G_k\right).$$
The kernel of $\rho_k$ is called the $k$-th (generalized) \emph{Johnson subgroup}, denoted by $\JG_k(S)$. In particular, $\mathcal{J}_2(S)$ is the Torelli group $\TG(S)$, which is the subgroup of $\MCG(S)$  consists of elements which acts trivially on $H_{1}(S)$,  and $\mathcal{J}_3(S)$ is the original Johnson subgroup, which is the subgroup of $\MCG(S)$ generated by Dehn twists along separating simple closed curves (\cite{Farb:2012, Johnson:1985jp}).

Fix a handlebody $W$ bounded by $S$, recall that if $f \in \MCG(S)$ lies in the handelbody group of $W$ if and only  if $f$ extends to a self-homeomorphism of $W$.  If there is a sub-compresionbody $W'$ lies in $W$ with $\partial_{+} W=\partial_{+} W'=S$,  such that  $f$  extends to a homeomorphism of  $W'$, then we say $f$  \emph{partially extends} to $W$. Biringer-Johnson-Minsky \cite{Biringer:2010ur} gave a  criterion whether  some powers of a pseudo-Asonov map $f$ partially extends  to a fixed handlebody $W$.

It seems that the  existence of a map in the Torelli subgroup of $\MCG(S)$ which does not extend to any handlebody bounded by $S$ is due to Johannson-Johnson and Casson, but the first proof in literature is Leininger-Reid \cite{Leininger:2002fe}.

For $g>2$ and sufficiently large $k$, Hain \cite{Hain:2008wb} showed the existence of a pseudo-Anosov map $f \in \mathcal{J}_k(S_{g})$ which does not extend to any handlebody $W$ bounded by $S_{g}$, using some highly sophisticated method. Independently, see also Jorgensen \cite{Jorgensen:2008tz}.

Using  handlebody graph,  Corollary 6.10 of Maher-Schleimer \cite{MaherSchleimer:2018} implies that

\begin{thm}\label{thm:notpartialextend} \cite{MaherSchleimer:2018}
For $S=S_{g}$ with $g\geqslant2$ and any $k \in \mathbb{Z}_{+}$, there exist a pseudo-Asonov map $\phi \in \mathcal{J}_k(S_{g})$ that does not partially extend to any handlebody bounded by $S$.
\end{thm}

We use quasi-homomorphisms to study the structure of the Johnson subgroups, we have the following

\begin{thm}\label{thm:generalization}
For a handlebody $H$ of genus $g>1$ with boundary $S$,  we have
$$\dim\QH(\JG_n(S),\HG\cap \JG_n(S))=\infty,\quad\text{for any}\quad n\geqslant2.$$
\end{thm}

It is interesting to note that the Johnson subgroup $\mathcal{J}_3(S_g)$ is finitely generated  when $g\geqslant4$ (\cite{Church2017}). It is a classical result that $\mathcal{J}_3(S_2)$ is not finitely generated (\cite{McCullough1986}).

\vspace{10pt}
Casson's invariant $\lambda$ is an integer valued invariant for 3-dimensional  integral homology spheres. From the standard genus $g$ Heegaard splitting $S^3=H_{+}\cup_{S} H_{-}$ of $S^3$,  for any $f \in \TG(S)$, we can cut $\mathbb{S}^3$ along $S$ to obtaining to handlebodies $H_
{+}$ and $H_{-}$, then we  re-glue them using $f$, the resulting 3-manifold $M_{f}$ is an integral homology sphere.

In a series of papers (see \cite{Morita:1989cr, Morita:1999cr} for more references), Morita studied the relation between the structure of the mapping class groups and invariants of 3-manifolds. In particular,  Morita \cite{Morita:1989cr} showed  that for every integral homology sphere $M$, there is  a $g \in \mathbb{N}$, such that $M$ can be obtained from  the genus $g$ Heegaard splitting of  $\mathbb{S}^3$  and some $f \in \mathcal{J}_3(S_{g}) \subset \TG(S_{g})$. Moreover,

\begin{thm}  [Morita \cite{Morita:1989cr}]\label{theorem:morita} For the genus $g$ surface $S_g$, the map $\lambda:\mathcal{J}_3(S_{g})\rightarrow \bbz$ is a homomorphism.
\end{thm}

So, in particular, Casson's invariant is a nonzero vector in $\QH(\JG_3(S),\HG \cap \JG_3(S)).$

\vspace{10pt}

\subsection{Quasi-invariants for Heegaard splittings}

We make the following definition:

\begin{defn} For a set $X$ with an equivalence relation $\sim$, a function $\varphi: X \rightarrow\bbr$ is called a \emph{quasi-invariant} on $X/\sim$ if there exists some $K\geqslant0$ such that
$$\abs{f(x)-f(y)}\leqslant K\quad\text{whenever}\quad x\sim y.$$
When $K=0$, a quasi-invariant becomes an invariant on $X/\sim$.
\end{defn}

The Dehornoy floor  for  conjugate classes of braids in $B_n$ is a quasi-invariant \cite{Malyutin:2005jq}. We give another example as follows.

\begin{example}
Take the unit disk $B^3$ in $\bbr^3=\bbr\times\bbc$, and consider the four points
$$A_1=\left(-\frac{\sqrt{2}}{2},\frac{\sqrt{2}}{2},0\right),\ A_2=\left(-\frac{\sqrt{2}}{2},-\frac{\sqrt{2}}{2},0\right),\ A_3=\left(\frac{\sqrt{2}}{2},\frac{\sqrt{2}}{2},0\right),\ A_4=\left(\frac{\sqrt{2}}{2},-\frac{\sqrt{2}}{2},0\right),$$
and the disks
$$D_1=\left\{(x,y,0)\ \Big|\ x^2+y^2=1,x\leqslant-\frac{\sqrt{2}}{2}\right\},\quad D_2=\left\{(x,y,0)\ \Big|\ x^2+y^2=1,x\geqslant\frac{\sqrt{2}}{2}\right\}.$$
Let $S_{0,4}$ be the $\partial B^3$ with $A_i$'s punctured. Let $\calt$ be isotopy classes of orientation-preserving diffeomorphisms of $B^3$ with $D_1\cup D_2$ fixed setwise, that is, $\calt=\pi_0(\mathrm{Diff}^+(B^3,D_1\cup D_2))$. Let $C$ be the curve $(0,\cos\theta,\sin\theta)$ and $\delta_C:B^3\rightarrow B^3$ be the function defined by
$$\delta_C(x,y,z)=\begin{cases}(x,y,z), & \text{if } \abs{x}\geqslant\frac12,\\ (x,y\cos h(x)+z\sin h(x),-y\sin h(x)+z\cos h(x)), &\text{if } \abs{x}\leqslant\frac12,\end{cases}$$
where $h(x)$ is a smooth monotonically increasing function such that $h(x)\equiv0$ for $x\leqslant-\frac12$ and $h(x)\equiv2\pi$ for $x\geqslant\frac12$. When restricted to $S$, $\delta_C$ is a Dehn twist along $C$.

The disk $D_3=\left\{x=0, y^2+z^2\leqslant1\right\}$ is the unique essential proper disk in $(B^3,D_1\cup D_2)$. Then $\calt$ is the stabilizer of the disk $D_3$ in $\MCG(S)$. By \cite[Theorem 1.1]{Bestvina:2007et}, $\QH(\MCG(S),\calt)$ is infinite dimensional.

Note that every double coset $\calt f\calt$ corresponds to a two-bridge link in the 3-sphere. Hence for any $f\in \MCG(S)\setminus\calt$ and $h\in\QH(\MCG(S),\calt)$, we have a numeric function $h(f)$, which is a quasi-invariant of two-bridge links.

\end{example}

Let $f_1,\cdots,f_n$ be quasi-invariants on a set $X$ with an equivalence relation $\sim$. $f_1,\cdots,f_n$ are \emph{linearly dependent} if there exist $(a_1,\cdots,a_n)\in\bbr^n\setminus\left\{0\right\}$ and $M>0$ such that
$$\abs{a_1f_1(x)+a_2f_2(x)+\cdots+a_nf_n(x)}<M,\qquad \forall\, x\in X.$$

Let $S$ be the boundary of the genus $g$ handlebody $H$,  and $\HG$ be the handlebody group. For $\phi \in \MCG(S)$, the double coset $\HG\phi\HG$ gives a genus $g$ Heegaard splitting of some 3-manifold.

\begin{cor} For each $g>1$, there exist infinitely many linearly independent quasi-invariants on the set of genus $g$ Heegaard splittings of closed orientable 3-manifolds.
\end{cor}

\begin{proof} It evidently follows from Theorem 1.1. \end{proof}

It has been the wish of many topologists to find interesting invariants for Heegaard splittings.  However, we think it more promising and reasonable to study quasi-invariants of Heegaard splittings. We hope that interesting quasi-invariants with geometric/topological meaning can be found.

\section{Preliminaries}

Let $S=S_{g}$ be a compact orientable surface of genus $g>1$, Harvey \cite{Harvey81} associated to  $S$  the following simplicial complex, which plays important roles  in low dimensional geometry and topology, Teichm\"uller theory, Kleinian groups  and  mapping class groups, see for example,   \cite{Masur:1999tc, Masur:2000tc} :

\begin{defn}
The \emph{curve complex} $\calc(S)$ is the simplicial complex
whose vertices correspond to the isotopy classes of simple closed curves in $S$ and whose  $k$-simplices correspond to the sets $\{ v_{0}, v_{0}, \ldots,v_{k}\}$ of distinct curves that have pairwise disjoint representatives.
\end{defn}

A celebrated theorem of
Masur-Minsky  \cite{Masur:1999tc} is the following:

\begin{thm}\emph{(Masur-Minsky)} $\calc(S)$ is $\delta$-hyperbolic for some $\delta\geqslant0$.
\end{thm}

A geodesic metric space $X$ is \emph{$\delta$-hyperbolic} ($\delta\geqslant0$) if, for any geodesic triangle in $X$, the closed $\delta$-neighborhood of any two edges contains the third edge. In a $\delta$-hyperbolic space, any two $(a,b)$-quasi-geodesics with the same end points are $c$ parallel, with the constant $c$ depends on $a,b$ and $\delta$ only. This fact will be used in Section 3.

Let  $H$ be an orientable handlebody of genus at least 2, and fix a homeomorphism from $\partial H$ to $S$, the \emph{disk complex} $\cald(H)$ is the sub-complex of $\calc(S)$ with vertexes consisting of simple closed curves which bounds a disk in $H$, and   $k$-simplices  defined from the simplicial structure of   $\calc(S)$. $\cald(H)$ is a connect complex,  but the inclusion map from $\cald(H)$ to   $\calc(S)$   is not a  quasi-isometric embedding \cite {Masurschleimer:2013}. The following theorem in \cite{Masur:2004tc} is useful for our purpose:

\begin{thm}\emph{(Masur-Minsky)} \label{thm:quasiconvex}The disk set $\cald(H)$ is quasi-convex in  $\calc(S)$.
\end{thm}

A subset $A$ in a geodesic metric space $X$ is \emph{quasi-convex} if there is some constant $R\geqslant0$ such that any geodesic with endpoints in $A$ is contained in the $R$-neighborhood of $A$.

For each $f \in\MCG(S)$, $f(\cald(H))$ is the disk set of another handlebody bounded by $S$. Now we have an
electrification of $\calc(S)$  by these connected subsets: for each disk set of a handlebody $W$ bounded by $S$, we adding a new vertex $v_{W}$ to $\calc(S)$, and we add a  length one edge from $v_{W}$ to every simple closed curve $c \in \calc^{0}(S)$  which bounds a disk in $W$,  the resulting complex is called the \emph{handlebody graph}  $\mathcal{HG}(S)$ \cite{MaherSchleimer:2018}.

\begin{thm}\emph{(Maher-Schleimer) \cite{MaherSchleimer:2018} } The handlebody graph  $\mathcal{HG}(S)$ is $\delta^\prime$-hyperbolic with infinite diameter for some $\delta^\prime\geqslant0$.
\end{thm}

\section{Infinite dimension of $\QH(\MCG(S),\HG)$}

In this section,we prove our theorems. First, we need a special pseudo-Anosov map in the mapping class group of $S$.

\subsection{A key theorem}

\begin{thm}\label{thm:key}
Let $H$ be a handlebody with genus $g>1$ and $S=\partial H$. There exists a pseudo-Anosov  map $\phi \in \MCG(S)$ that satisfy the following property: let $\ell$ be its quasi-axis in the curve complex $\calc(S)$, then for any $B >0$,  there is an $L>0$ such that there exist no length $L$ segment $\alpha$ of $\ell$ and $g \in \MCG(S)$ with $g(\alpha)$ contained in the $B$-neighborhood of the disk set $\cald(H) \subset \calc(S)$.
\end{thm}

\begin{proof}
We fix a projection map  $\Pi: \calc(S) \rightarrow\mathcal{HG}(S)$, where for any $c \in \calc^{0}(S)$,  $\Pi(c)$ is $v_{W}$ for a  handlebody $W$ such that $c$ bounds a disk in $W$.  Note that  $\Pi$ is  well-defined up to diameter one sets.  We also  take two handlebodies $V$ and $W$ bounded by $S$ such that $d_{\mathcal{HG}}(v_{V},v_{W})$ is sufficiently large.  Now let $P_{W}(V)$ be any disk in $W$ such that
$d_{\calc(S)}(\cald(V),\cald(W))=d_{\calc(S)}(\cald(V),P_{W}(V))$,  where $d_{\calc(S)}(A,B)$ for two subsets $A$ and $B$ in $\calc(S)$ we mean the minimal distance of two points $a, b$ in  $A$ and $B$ respectively. Note that since $d_{\mathcal{HG}}(v_{V},v_{W})$ is sufficiently large, in particular $d_{\calc(S)}(\cald(V),\cald(W))$ is sufficiently large, and by Theorem \ref{thm:quasiconvex}, $P_{W}(V)$ is coarsely defined in $\calc(S)$ with constant depends by the hyperbolicity constant of $\calc(S)$ and the quasi-convexity constant of the disk set into $\calc(S)$.  Similarly we can define $P_{V}(W)$ which is the boundary of a disk in $V$.

Now take a pseudo-Anosov map $f \in \MCG(S)$ which extends to $V$ (i.e., $f$ lies in the handlebody of $V$) such that $d_{\calc(S)}(P_{V}(W), f(P_{V}(W))$ is sufficiently large. We also choose a  pseudo-Anosov map $g$ lies in the handlebody group of $W$ with similar property.

Note that $\MCG(S)$ acts naturally on the handlebody graph $\mathcal{HG}$. Let $\phi=gf$, and  $A$ be a geodesic in $\mathcal{HG}$ which connecting $v_{V}$ to $v_{W}$, and $B$ be a geodesic in $\mathcal{HG}$ which connecting $c_{W}$ to $\phi(v_{V})=gf(v_{V})=g(v_{V})$.  Then we have an infinite length piecewise geodesic path $\alpha=\cdots \phi^{-1}(A)\cup \phi^{-1}(B)\cup A\cup B \cup \phi(A) \cup \phi(B)\cup \phi^{2}(A) \cup \phi^{2}(B)\cdots$ in $\mathcal{HG}$. Since $\phi(\alpha)=\alpha$, then by Claim \ref{claim:quasigeodesic} below,  $\phi$ is a pseudo-Asonov map in $\MCG(S)$.

\begin{claim}\label{claim:quasigeodesic} $\alpha$ is a quasi-geodesic in $\mathcal{HG}$ of infinite length.
\end{claim}

Proof of Claim \ref{claim:quasigeodesic}: We take an oriented geodesic $A'$ in $\calc(S)$ which connecting  $P_{V}(W)$ to  $P_{W}(V)$, and take a geodesic  $C'$ in $\calc(S)$ which connecting  $P_{W}(V)$ to  $g(P_{W}(V))$, now let $B'$  be $g(A')$ but with inverse orientation, then it  is a geodesic connecting $g(P_{W}(V))$ to $g(P_{V}(W))$.  Let $E'$ be an oriented geodesic connecting  $P_{V}(W)$ to  $g(P_{W}(V))$, and $F'$ be an oriented geodesic connecting  $P_{V}(W)$ to  $g(P_{V}(W))$. By hyperbolicity of  $\calc(S)$ and quasi-convexity of $\cald(W)$ in  $\calc(S)$, we have there is a constant $k$, such that except the $k$-neighborhood  of $P_{W}(V)$, any point in $A' \cup C'$ lies in a bounded neighborhood of $E'$. By the same reason, most part of $C' \cup B'$ lies in in a bounded neighborhood of $F'$. All these together, we have $A' \cup C' \cup B'$ is a quasi-geodesic in $\calc(S)$.

Let $G'$ be a geodesic connecting  $P_{V}(W)$ to $f(P_{V}(W))$. Similar, we have $\phi^{-1}(B')\cup G' \cup A'$ is a quasi-geodesic in $\calc(S)$. Now we connect all of
$\{\phi^{i}(A')\}^{i=\infty}_{i=-\infty}$, $\{\phi^{i}(B')\}^{i=\infty}_{i=-\infty}$, $\{\phi^{i}(C')\}^{i=\infty}_{i=-\infty}$ and $\{\phi^{i}(G')\}^{i=\infty}_{i=-\infty}$
into an infinite length path $\alpha'$.  $\alpha'$ is piece-wisely geodesic,  but any three consecutive geodesic segment of it with middle part $\phi^{i}(C')$ or $\phi^{i}(G')$ is a uniform quasi-geodesic. Since the length of $C'$ and $G'$ is large, then it is standard in large scale geometry that $\alpha'$ is a quasi-geodesic.

 Note that since we electrification quasi-convexity subset to obtain  $\mathcal{HG}$, $\Pi(\alpha')$ is a re-parameterized  quasi-geodesic in $\mathcal{HG}$ for a uniform constant (see also \cite{MaherSchleimer:2018} ). Now $\alpha$  is parallel to $\Pi(\alpha')$, so $\alpha$ is a quasi-geodesic in $\mathcal{HG}$ of infinite length. This ends the proof of Claim \ref{claim:quasigeodesic}.


Now take a quasi-axis  $\ell$ of $\phi$ in the curve complex $\calc(S)$,  we claim for any $B\in \mathbb{R}_{\geq 0}$, there is an $L\in \mathbb{R}_{\geq 0}$ such that for any handlebody $Z$ bounded by $S$, each component of $\ell \cap N_{B}(\cald(Z))$ has diameter bounded above by $L$. Suppose otherwise, there is a $B$, for any $L\in \mathbb{R}_{\geq 0}$ there is a handlebody $Z_{L}$ such that $diam (\ell \cap N_{B}(\cald(Z_{L}))) \geq L$. Fix a point $P \in \ell$, up to translating these handlebodies by $\phi$, we have a sequences of handlebodies $Z_{L}$, $L \rightarrow \infty$, such that $P \in  N_{B'}(\cald(Z_{L}))$ for another $B' \in \mathbb{R}_{\geq 0}$ and $diam (\ell \cap N_{B'}(\cald(Z_{L}))) \geq L$, this is a contradiction to that  $\Pi(\ell)$ is an infinite length quasi-geodesic which is parallel to $\alpha$.

\end{proof}

\subsection{Proof of main theorems}

With Theorem \ref{thm:key} in hand, then the proof of Theorem \ref{thm:main} is standard after Bestvina-Fujiwara \cite{Bestvina:2002dr,Bestvina:2007et}, we just outline the proof,  for more details see  \cite{Bestvina:2002dr,Bestvina:2007et}.

\begin{proof}[Proof of Theorem \ref{thm:main}]We fix a disk $d_0$ in the disk set  $\cald(H)$ of the handlebody $H$.

Step 1: Let $\omega$, $\alpha$ be two oriented finite paths in  $\calc(S)$, for a mapping class  $g \in \MCG(S)$, we say $g(\omega)$ is a \emph{translation} of $\omega$. Let $|\alpha|_{\omega}$ denote the maximal number of non-overlapping translations of $\omega$ in $\alpha$. Fixed $W$ be a nonzero integral smaller than the length $|\omega|$ of $\omega$. For two vertices $x,y$ in $\calc^{0}(S)$, we can define

$$C_{\omega, W}(x,y)\triangleq d_{\calc(S)}(x,y)-\inf_{\alpha}(|\alpha|-W |\alpha| _{\omega}),$$   where $\alpha$ ranging over all paths from $x$ to $y$.

\begin{lemma}[Fujiwara \cite{Fujiwara:1998}]\label{lem:quasigeodesic}  Fix $W \in \mathbb{Z}_{+}$, if $\alpha$ is a path  realizes the infimum in the definition of $C_{\omega, W}(x,y)$ above, then  $\alpha$ is a quasi-geodesic with constants decreasing on $|\omega|$.
\end{lemma}

Fix $W \in \mathbb{Z}_{+}$ smaller than $|\omega|$, we can define

$$h_{\omega: } \MCG(S) \rightarrow \bbr$$ as
$$h_{\omega: }(g)\triangleq C_{\omega, W}(d_0, g(d_0))-C_{\omega^{-1}, W}(d_0, g(d_0)).$$

Then  from Proposition 3.10 of \cite{Fujiwara:1998},  $h_{\omega}$ is a quasi-homomorphism with  defect uniformly bounded above  independent of  $\omega$.
\vspace{10pt}

Step 2: Choose a pseudo-Anosov map $\phi$ as in  Theorem \ref{thm:key}, we also choose another pseudo-Anosov map $\psi$ such that $\{\Lambda_{+}(\phi),\Lambda_{-}(\phi)\}\cap\{\Lambda_{+}(\psi),\Lambda_{-}(\psi)\}=\emptyset$, where $\Lambda_{+}(\phi)$ ($\Lambda_{-}(\phi)$) is the stable (unstable) measured lamination where $\phi$ fixed in the projective measured lamination space of $S$.   After passing to high powers of $\phi$ and $\psi$, we may assume $\phi$ and $\psi$ generates a Schottky subgroup $F$ of $\MCG(S)$. In particular, every element in $F$ except the identity is pseudo-Anosov.  There is a $F$-equivariant map $\rho$ from a Cayley graph of $F$ to $\calc(S)$, which is a quasi-isometric embedding.  So for each $f \in F$,  we can take an axis of it in the Cayley graph and then using the map $\rho$,  we get a uniform quasi-axis for the action of $f$ on $\calc(S)$.

\vspace{10pt}

Step 3: Choose $n_{i}, m_{i}, k_{i}, l_{i} \in \mathbb{Z}$, such that

 $$0 \ll n_{1} \ll m_{1} \ll k_{1} \ll l_{1} \ll n_{2} \ll m_{2} \ll k_{2} \ll l_{2} \cdots$$

 and we can define $f_{i}=\phi^{n_i}\psi^{m_i}\phi^{k_i}\psi^{-l_i}$. By Bestvina-Fujiwara \cite{Bestvina:2002dr}, any two elements in the set $\{f_{i}^{\pm}\}^{\infty}_{i=1}$ are  not equivalent as actions on $\calc(S)$.  It turns out that if we choose $\omega_{i}$  be a  sufficiently long segment of a uniform quasi-axis of $f_{i}$, and fix $W\ \in \mathbb{Z}$ which is smaller than each $|\omega_{j}|$, then the quasi-homomorphisms $h_{\omega_{i}}$ are linearly independent in $\QH(\MCG(S))$ \cite{Bestvina:2002dr}. For the  finite collection of cyclic subgroups $C_{1}, C_{2}, \ldots, C_{n} \subset \MCG(S)$ generated by $g_{1}, g_{2}, \ldots, g_{n}$,  we can delete finitely many elements  in our set  $\{f_{i}\}^{\infty}_{i=1}$, such that for each of the remaining $f_{j}$, $f^{\pm}_{j}$ is not equivalent to any $g_{i}$, $1 \le g  \le n$, then by Proposition  5.1 of \cite{Bestvina:2002dr}, $h_{\omega_{j}}$ is zero on each $C_{i}$.

 \vspace{10pt}
 Step 4:  We can choose $\omega_{j}$ in Step 3 such that  it contains a copy of a segment of $\ell$ which has length $L$ sufficiently large. From the choice of
 $\ell$, we have for any $f \in \MCG(S)$, $f(\omega_{j})$  is not in the $B$-neighborhood of the disk set  $\cald(H)$. For any $g$ in the handlebody group $\HG$, $g(d_{0}) \in \cald(H)$, then by Theorem  \ref{thm:quasiconvex},  $\cald(H)$ is a quasi-convex subset of  $\calc(S)$, any uniform quasi-geodesic  connecting $d_{0}$ and $g(d_{0})$ lies in the  $B$-neighborhood of the disk set  $\cald(H)$. Then by the definition of $C_{\omega, W}$, we have both of  $C_{\omega_{j}, W}(d_{0},g(d_{0}))$ and $C_{\omega^{-1}_{j}, W}(d_{0},g(d_{0}))$ are zero,  so $h_{\omega_{j}}(g)$ is zero.
\end{proof}

\begin{proof}[Proof of Theorem \ref{thm:generalization}]
 Fix $\phi=\phi_{1}$ in  as in  Theorem \ref{thm:key}, for any $k \in \mathbb{Z}$, since $\JG_k(S)$ is a nontrivial normal subgroup $\MCG(S)$, we can take three  pseudo-Anosov  maps $\phi_{1}$,  $\phi_{2}$ and $\phi_{3}$ in  $\JG_k(S)$ that are pairwise independent, i.e.,  $\{\Lambda_{+}(\phi_{i}),\Lambda_{-}(\phi_{i})\}\cap\{\Lambda_{+}(\phi_{j}),\Lambda_{-}(\phi_{j})\}=\emptyset$ for $i\neq j$.  After passing to high powers of $\phi_{i}$, we may assume $\{\phi_{i}\}^{3}_{i=1}$ generate a rank three Schottky subgroup $F$ of $\MCG(S)$.  In particular, $\phi=\phi_{1}^{a}\phi_{2}^{b}\phi_{1}^{-a}\phi_{3}^{c}$ is a
 pseudo-Anosov map in $\JG_k(S)$ for $a\ll b\ll c$ in $\mathbb{Z}$. Moreover, using the  quasi-isometric $F$-equivariant map  from a Cayley graph of $F$ to $\calc(S)$, we get that a quasi-axis of  $\phi$ on $\calc(S)$ contains a long segment of $\ell$.

 With this new $\phi \in\JG_k(S)$, we choose a new $\psi \in \JG_k(S)$, then re-work Steps 1-4 in the proof of Theorem \ref{thm:main}, we get Theorem \ref{thm:generalization}.
\end{proof}

\section*{Acknowledgement}
The proof of Theorem  \ref{thm:key} using handlebody graph  was suggested by Saul Schleimer when Jiming Ma visited  University of Warwick in  2013, the authors would like to thank Saul for his help and interest in this paper. Jiming Ma was partially supported by NSFC   11771088 and Jiajun Wang was partially supported by NSFC 11131008.
\bibliographystyle{amsplain}

\end{document}